\documentclass[reqno, 11pt]{amsart}
\usepackage{epsfig,amsfonts,amsthm,amssymb,latexsym,amsmath,enumerate,hyperref,color}
\newtheorem{theorem}{Theorem}
\newtheorem{corollary}[theorem]{Corollary}
\newtheorem{lemma}[theorem]{Lemma}
\newtheorem{definition}{Definition}

\title[Olver's Method for Solving Roots of $p$-Adic Polynomial Equations]
{Olver's Method for Solving Roots of $p$-Adic Polynomial Equations}

\author[J. F. T. Rabago]{Julius Fergy T. Rabago}
\address{Julius Fergy Tiongson Rabago, Department of Mathematics and Computer Science,
College of Science, University of the Philippines Baguio, Baguio City 2600, Benguet, PHILIPPINES} \email{jfrabago@gmail.com}



\date{\today}
\begin{document}
\maketitle

\begin{abstract}
Let $\mathbb{Z}_p[x]$ be the set of all functions whose coefficients are in the field of $p$-adic integers $\mathbb{Z}_p$.
This work considers a problem of finding a root of a polynomial equation $P(x)=0$ where $P(x)\in\mathbb{Z}_p[x]$. 
The solution is approached through an analogue of Olver's method for finding roots of polynomial equations $P(x)=0$ in $\mathbb{Z}_p$.
\end{abstract}

%---------------------------------------------------------- INTRODUCTION ----------------------------------------------------------
\section{Introduction}
A $p$-adic number, which was first introduced in 1897 by Kurt Hensel, is an extension of the field of rationals $\mathbb{Q}$ such that congruences modulo powers of a fixed prime $p$ are related to proximity in the so called  \emph{$p$-adic metric}.
This metric provides a totally different notion of `closeness' or absolute value. 
To be precise, two $p$-adic numbers are said to be close when their difference is divisible by a high power of $p$, meaning to say, the higher the power the closer they are.
This notion of closeness obviously shows that the field of $p$-adic numbers denoted by $\mathbb{Q}_p$ extends the ordinary arithmetic in $\mathbb{Q}$ in a way different from the extension of $\mathbb{Q}$ to the real and complex number systems $\mathbb{R}$ and $\mathbb{C}$.  
Moreover, the creation of $p$-adic numbers were actually due to an attempt to bring the ideas and techniques of power series methods into number theory. 
So at first, just like many other fields of mathematics, $p$-adic numbers are considered as an exotic part of pure mathematics without any application. 
However, various applications to other fields of mathematics, especially in analysis and algebraic geometry, and related areas in the applied sciences (e.g., physics and bioinformatics) have been recently proposed and discovered.
For instance, the field of $p$-adic analysis essentially provides an alternative form of calculus (cf. \cite{baker, vladimirov}) and the $p$-adic numbers appear to have some applications in modeling DNA sequences and genetic codes (cf. \cite{dragovich}). 
In 1968, Monna and van der Blij proposed to apply $p$-adic numbers to physics and in 1972, Beltrametti and Cassinelli investigated a model of $p$-adic valued quantum mechanics from the positions of quantum logic. 
$p$-adic numbers were also found to have some sort of applications in quantum physics since 1980's (cf. \cite{razikov}).
For more applications in the physical sciences, especially in the construction of physical models (e.g., string theory, quantum mechanics, quantum cosmology and dynamical systems), we refer the readers to \cite{brekke} and \cite{vladimirov}.
One may also want to consult an article of Razikov \cite{razikov} (and see the references therein) for a popular introduction to the theory of $p$-adic numbers.

On the other hand, in a purely theoretical aspect, Hensel\rq{}s lemma provides sufficient conditions for the existence of roots in $\mathbb{Z}_p$ of polynomials in $\mathbb{Z}_p[x]$.
A classical application of this lemma deals with the problem of finding roots of a $p$-adic number $a$ in $\mathbb{Q}_p$ and this 
was in fact the subject of several recent investigations about $p$-adic numbers.  
In 2010, for instance, Knapp and Xenophontos \cite{knapp} showed how classical root-finding methods from numerical analysis can be used to calculate inverses of units modulo prime powers. 
In the same year, Zerzaihi, Kecies and Knapp \cite{zer1} applied some classical root-finding methods, such as the fixed-point method, in finding square roots of $p$-adic numbers through Hensel\rq{}s lemma.
In 2011, Zerzaihi and Kecies \cite{zer2} used secant method to find the cubic roots of $p$-adic numbers. 
These authors \cite{zer3} then applied the Newton method to find the cubic roots of $p$-adic numbers in $\mathbb{Q}_p$. 
A similar problem also appeared in \cite{ignacio} wherein Ignacio et al. computed the square roots of $p$-adic numbers via Newton-Raphson method. 
In \cite{bacani} and \cite{rabago}, it was observed that none of these aforementioned works have considered the problem of finding roots of a general $p$-adic polynomial.
So, motivated by this problem, Bacani and the author proposed in \cite{bacani} an analogue of Steffensen's method in finding roots of a general $p$-adic polynomial equation $f(x)=0$ in $\mathbb{Z}_p$.
In \cite{rabago}, on the other hand, the author described an analogue of Halley's method for approximating roots of $p$-adic polynomial equations $f(x)=0$ in $\mathbb{Z}_p$.
A related study which examines a $p$-adic analogue of Newton-Raphson's method was also considered in \cite{rabago0}. 
In this work we offer another approach in solving a root-finding problem $f(x)=0$ in the $p$-adic case through Olver's method.
We shall show that the method, even in the $p$-adic case, is more efficient compare to Steffensen's and Newton's method in terms of convergence property.
Moreover, we shall see that the $p$-adic analogue of Olver's method is faster than Halley's method in computing $2$-adic roots of polynomial equations in $\mathbb{Z}_p$. 

The rest of the paper is organized as follows. 
In Section 2, we discuss some concepts about the field $\mathbb{Q}_p$ and present some ideas about continuity and differentiability of functions on $\mathbb{Q}_p$. 
Our main theorem is presented in Section 3 along with its proof and lastly, in Section 4, a short conclusion about our present work is stated.

\section{Preliminaries}

In this section we discuss some basic properties of $\mathbb{Q}_p$. 
A more formal treatment of the topic can be found in a book of Katok on $p$-adic numbers \cite{katok}.

\subsection{\bf The field $\mathbb{Q}_p$.}

We start with the definition of the $p$-adic norm $|\cdot|_p$ and the $p$-adic valuation $v_p$ on $\mathbb{Q}$.

\begin{definition}
Let $p$ be a fixed prime. 
The $p$-adic norm $|\cdot|_p:\mathbb{Q}\rightarrow \{p^n : n\in\mathbb{Z}\} \cup \{0\}$ is defined as follows:
\[
\forall x\in\mathbb{Q}: |x|_p =\left\{
\begin{array}{ll}
p^{-v_p(x)}, 	&\text{if} \ x\neq 0,\\[0.5em]
0, 			&\text{if}\ x=0,
\end{array}
\right.
\]
where $v_p$ is the $p$-adic valuation defined by $v_p(x)=\max\{r\in \mathbb{Z}:p^r \mid x\}$. 
This norm induces the so-called $p$-adic metric $d_p$ given by
\begin{align*}
d_p: \mathbb{Q} \times \mathbb{Q} 	\quad &\longrightarrow \quad \mathbb{R}^+\\
(x,y) 							\quad &\longmapsto \quad d_p(x,y)=|x-y|_p.
\end{align*}
\end{definition}

The $p$-adic norm $|\cdot|_p$ satisfies the following important properties (cf. \cite{katok}):
\begin{enumerate}
\item[(i)] $|xy|_p=|x|_p|y|_p$;
\item[(ii)] $|x+y|_p \leq \mbox{max}\{|x|_p,|y|_p\}$, where equality holds if $|x|_p\neq |y|_p$; and
\item [(iii)] $\left|\dfrac{x}{x}\right|_p=\dfrac{|x|_p}{|y|_p}$.
\end{enumerate}

Now the field of $p$-adic numbers $\mathbb{Q}_p$ is formally defined as follows.

\begin{definition}
\label{fieldQp}
The field $\mathbb{Q}_p$ of $p$-adic numbers is the completion of $\mathbb{Q}$ with respect to the $p$-adic norm $|\cdot|_p$. 
The elements of $\mathbb{Q}_p$ are equivalence classes of Cauchy sequences in $\mathbb{Q}$ with respect to the extension of the $p$-adic norm defined as
\[
|a|_p=\lim_{n \rightarrow \infty} |a_n|_p,
\]
where $\{a_n\}$ is a Cauchy sequence of rational numbers representing $a \in \mathbb{Q}_p$.
\end{definition}

The following theorem provides a way to write a $p$-adic number in a unique representation.

\begin{theorem}
\label{canonical}
Given a $p$-adic number $a \in \mathbb{Q}_p$, 
there is a unique sequence of integers $(a_n)_{n\geq N}$, with $N=v_p(a)$, such that $0\leq a_n \leq p-1$ for all $n$ and 
\[
a=a_Np^N + a_{N+1}p^{N+1} + \cdots +a_np^n+\cdots
=\sum_{k=N}^{\infty}a_i p^i.
\]
\end{theorem}

With such representation as above, a $p$-adic number is naturally defined as a number $a\in\mathbb{Q}_p$ whose canonical expansion contains only nonnegative powers of $p$.
The set of $p$-adic integers is denoted by $\mathbb{Z}_p$ and is given by
\[
\mathbb{Z}_p=\left\{a\in \mathbb{Q}_p : a = \sum_{i=0}^{\infty} a_ip^i, \ 0 \leq a_i \leq p-1\right\} =\left\{a \in \mathbb{Q}_p : |a|_p\leq1\right\}.
\]

\begin{definition}
The group of invertible elements in $\mathbb{Z}_p$ (or the group of $p$-adic units) denoted by $\mathbb{Z}_p^{\times}$ is given by
\[
\mathbb{Z}_p^{\times} 
=\left\{ a\in \mathbb{Z}_p : a = \sum_{i=0}^{\infty} a_ip^i, \ a_0 \neq 0\right\} 
=\left\{a \in \mathbb{Q}_p : |a|_p=1\right\}.
\]
\end{definition}

By virtue of Theorem \ref{canonical} we may write, in an alternative way, a $p$-adic number in terms of their $p$-adic valuation.

\begin{corollary}
Let $a \in \mathbb{Q}_p$. 
Then, $a=p^{v_p(a)}u$ for some $u \in \mathbb{Z}_p^{\times}$.
\end{corollary} 

The following lemma will be central to our discussion.

\begin{lemma}
\label{estimate}
Let $a,b\in \mathbb{Q}_p$. Then, 
\[
a\equiv b \ ({\rm mod}\ p^m) \quad \Longleftrightarrow \quad |a-b|_p\leq p^{-m}.
\]
\end{lemma}

%----------------------------------------------------
\subsection{\bf Functions over $\mathbb{Q}_p$.}

In this section we discuss some fundamental concepts on the analysis of functions defined over $\mathbb{Q}_p$.

Let $X \subset \mathbb{Q}_p$. 
A function $f : X \rightarrow \mathbb{Q}_p$ is said to be continuous at $a \in X$ 
if for each $\varepsilon > 0$ there exists a $\delta > 0$ such that if $|x - a|_p < \delta$, then $|f(x) - f(a)|_p < \varepsilon$. 
A function $f$ is said to be continuous on $E \subseteq X$ if $f$ is continuous for every $a \in E$.
Also, let $a\in X$ be an accumulation point of $X$. 
Then, the function $f$ is differentiable at $a$ if the derivative of $f$ at $a$, defined by
\[
f'(a)=\lim_{x\rightarrow a} \frac{f(x)-f(a)}{x-a}
\]
exists. 
In general, $f$ will be differentiable on $X$ if $f'(a)$ exists at all $a \in X$.

As a simple example, any polynomial function $f$ in $\mathbb{Q}_p[x]$ is contiuous and differentiable at every $a\in\mathbb{Q}_p$ (cf. \cite{rabago}).

%---------------------------------
\subsection{\bf $p$-Adic roots.}

The following are some important results in the study of $p$-adic roots.
The following theorem can be found, e.g., in \cite[p. 48]{robert}.

\begin{theorem}[Hensel's lemma] 
Let $F$ be a polynomial of degree $q\in\mathbb{N}$ whose coefficients are $p$-adic integers, i.e.,
\[
F(x)=c_0 + c_1x+c_2x^2+\ldots+c_qx^q \in \mathbb{Z}_p[x]
\] 
and 
\[
F'(x)=c_1 + 2c_2x+3c_3x^2+\ldots+qc_qx^{q-1}
\] 
be its derivative. 
Suppose for $\overline{a_0} \in \mathbb{Z}_p$ we have $F(\overline{a_0}) \equiv 0 \ ({\rm mod}\ p)$ and $F'(\overline{a_0})  \not\equiv 0 \ ({\rm mod}\ p)$.
Then, there is a unique $a \in \mathbb{Z}_p$ such that $F(a) = 0$ and $a \equiv \overline{a_0} \ ({\rm mod}\ p)$.
\end{theorem}

\begin{theorem}
A polynomial with integer coefficients has a root in $\mathbb{Z}_p$ if and only if it has an integer root modulo $p^m$ for any $m \in\mathbb{N}$.
\end{theorem}

For the proof of the above theorems, one can consult a text on $p$-adic analysis by Katok \cite{katok}.

Having these ideas understood, we are now ready to present and validate our main results in the next section.

%---------------------------------------------------------- MAIN RESULTS ----------------------------------------------------------
\section{Main Results}
Olver's method is a root-finding algorithm which is iteratively defined by 
\begin{equation}
\label{recursion}
x_{n+1}=x_n-\frac{f(x_n)}{f'(x_n)}-\frac{1}{2}\left\{ \frac{[f(x_n)]^2f''(x_n)}{[f'(x_n)]^3}\right\},\qquad \forall n\in\mathbb{N}_0. \tag{OM}
\end{equation}
This method is cubically convergent; that is, for every sequence $(x_n)_{n\in \mathbb{N}_0}$ generated through the recurrence \eqref{recursion}, we have 
\[
\lim_{n\to \infty} \frac{\left|x_{n+1} - L \right|}{\left|x_n - L \right|^3} = {\rm const.} > 0
\]
where $L$ is the limit of the sequence $(x_n)_{n\in \mathbb{N}_0}$.
Hence, it is faster compare to the well-known Newton's (or sometimes called Newton-Raphson's) method.
The algorithm is named after Frank W. J. Olver who first introduced the method in \cite{olver}.

In this section, as stated in the introduction, we are interested in finding a root of the polynomial equation $f(x)=0$ with $f(x) \in \mathbb{Z}_p[x]$ through Olver's method.
Throughout the discussion we denote, as usual, the derivative of $f$ as $f'$ and for simplicity, we use $a\equiv_p b$ to denote the congruence relation $a\equiv b$ (mod $p$).

Our main result is given as follows.
%---------------- MAIN THEOREM
\begin{theorem}
\label{main}
Let $x_0 \in \mathbb{Z}$ such that $f(x_0)\equiv_p 0$ and $f'(x_0)\not\equiv_p 0$ where $f(x)\in\mathbb{Z}_p[x]$ is a polynomial of degree $q\in\mathbb{N}$. 
Define the sequence $(x_n)_{n\in\mathbb{N}_0}$ recursively by Olver's iterative formula \eqref{recursion}.
Then, we have the following results:
\begin{enumerate}
\item[{\rm (i)}]  
	\[
	\forall n\in \mathbb{N}_0: \ x_n\in\mathbb{Z}_p,\quad f'(x_n)\not\equiv_p0, \quad
	f(x_n) \equiv
	\left\{
	\begin{array}{ll}
	0 \left({\rm mod} \ p^{3^n}\right)&\text{if}\ p \neq 2,\\[1em]
	0 \left({\rm mod} \ p^{4^n}\right)&\text{if}\ p=2.
	\end{array}
	\right.
	\]
\item[{\rm (ii)}] The sequence $(x_n)_{n\in\mathbb{N}_0}$ generated by the recursion \eqref{recursion} converges to a unique zero $\xi\equiv_p x_0$ of $f$ in $\mathbb{Z}_p$.

\item[{\rm (iii)}] The convergence in {\rm (ii)} is cubic for $p\neq2$ and is of fourth-order for $p=2$.   
\end{enumerate}
\end{theorem}

Before we prove the above results, we first establish the following lemma which will be central to our proof.

%----------------- LEMMA
\begin{lemma}
\label{lemma}
Let $a,y\in\mathbb{Z}_p$ and suppose that $y\equiv_{p^m}0$ for some positive integer $m$.
Then, 
\[
f(a+y)\equiv_{p^{km}}\sum_{j=0}^{k-1}\frac{f^{(j)}(a)}{j!}y^j, \qquad\forall k \in \{1,2,\ldots,{\rm deg}(f)\},
\]
where ${\rm deg}(f)$ denotes the degree of the polynomial $f\in\mathbb{Z}_p[x]$.
\end{lemma}

\begin{proof}
Let $f\in\mathbb{Z}_p[x]$ with ${\rm deg}(f)=q \in \mathbb{N}$ and $k\in\{1,2,\ldots,q\}$. 
Further, assume that $a,y\in\mathbb{Z}_p$ and suppose $y\equiv_{p^m}0$ for some positive integer $m$.
Using the well-known Taylor expansion formula (TEF) to $f(a+y)$, we get
\[
f(a+y) 
%= \sum_{j=0}^q \frac{f^{(j)}(a)}{j!}y^j
=f(a)+\cdots+\frac{f^{(k-1)}(a)}{(k-1)!} y^{k-1}+y^k\sum_{j=0}^{q-k} \frac{f^{(j+k)}(a)}{(j+k)!}y^j.
\]
But, by assumption, $y\equiv_{p^m}0$ for some positive integer $m$. 
Hence, $y^k\equiv_{p^{km}}0$ for every $k$. 
Thus we have
\[
f(a+y)\equiv_{p^{km}} f(a)+\cdots+\frac{f^{(k-1)}(a)}{(k-1)!}y^{k-1},
\]
which is desired.
\end{proof}
%
%---------------- REMARK
By above lemma, one easily finds that $f(a+y)\equiv_{p^m}f(a)$, $f(a+y)\equiv_{p^{2m}}f(a)+f'(a)y$ and $f(a+y)\equiv_{p^{3m}}f(a)+f'(a)y+f''(a)y^2/2$. 

Now, we are in the position to prove Theorem \ref{main}.
From here on we assume $f$ to take the form $f(x)=a_0+a_1x+a_2x^2+\ldots+a_qx^q \in \mathbb{Z}_p[x]$.

\begin{proof}[{\bf Proof of Theorem \ref{main}}]
Throughout the proof we write $a \equiv_{p^m} b$ when we wish to say that $|a-b|_p \leqslant p^{-m}$.
Let $x_0 \in \mathbb{Z}$ such that $f(x_0)\equiv_p 0$ and $f'(x_0)\not\equiv_p 0$. 
Furthermore, define the sequence $(x_n)_{n\in\mathbb{N}_0}$ recursively by \eqref{recursion}. 
We first prove (i).\\

%---------------------------------------- FIRST PART
\underline{Proof of i.}
The proof of this part proceeds by induction on $n$. 
We show that $f(x_n) \equiv_{p^{3^n}}0$ for $p\neq 2$ and $f(x_n) \equiv_{p^{4^n}}0$ for $p=2$ and $f'(x_n)\not\equiv_p 0$ for all $n \in \mathbb{N}_0$. 
First we note that since $x_0 \in \mathbb{Z}$, then $x_0\in\mathbb{Z}_p$.
%----------------------------------------- BASIS STEP
So for the basis step we have, by definition,  
\[
f(x_1) = f\left( x_0 -\frac{f(x_0)}{f'(x_0)}-\frac{1}{2}\left\{ \frac{[f(x_0)]^2f''(x_0)}{[f'(x_0)]^3}\right\}\right)
=:f(x_0+y_0).
\]
In view of Lemma \ref{lemma}, we have
\[
f(x_0+y_0) \equiv_{p^3} f(x_0)+f'(x_0)y_0 + \frac{f''(x_0)}{2}y_0^2.
\]
But, the right hand side of the above relation can be simplified as
\[
\frac{[f(x_0)]^3 [f''(x_0)]^2 \left\{ 4 [f'(x_0)]^2+f(x_0)f''(x_0)\right\}}{8 [f'(x_0)]^6}=:\frac{U}{V}.
\]
Clearly, for $p\neq 2$ and by the fact that $f(x_0)\equiv_p 0$ and $f'(x_0)\not\equiv_p 0$, it follows that 
\[
f(x_1) \equiv_{p^3} f(x_0+y_0) \equiv_{p^3} 0.
\]
Now from the form of $f$, we have $f''(x)=\sum_{j=0}^{q-2}=(j+1)(j+2)a_{j+2}x^j$.
But since $2\mid(j+1)(j+2)$, then $f''(x) \equiv_2 0$.
So for $p=2$, we see that $U \equiv_{2^7} 0$ and $V\equiv_{2^3}0$.
Hence, $f(x_1) \equiv_{p^4} 0$ for $p=2$. 
Similarly, since $f'(x_1) = f'(x_0+y_0)$, then by Lemma \ref{lemma}, we have $f'(x_1) \equiv_p f'(x_0)$.
We have just shown that the result holds for $n=1$. 
%---------------------------------------- INDUCTION HYPOTHESIS
Now, for the induction hypothesis, we suppose that for some $n_0\in\mathbb{N}$ the following results hold:
\[
	\forall n \geq n_0: \ x_n\in\mathbb{Z}_p,\quad f'(x_n)\not\equiv_p0, \quad
	f(x_n) \equiv
	\left\{
	\begin{array}{ll}
	0 \left({\rm mod} \ p^{3^n}\right)&\text{if}\ p \neq 2,\\[1em]
	0 \left({\rm mod} \ p^{4^n}\right)&\text{if}\ p=2.
	\end{array}
	\right.
\]
Let
\[
y_n:=-\frac{f(x_n)}{f'(x_n)}-\frac{1}{2}\left\{ \frac{[f(x_n)]^2f''(x_n)}{[f'(x_n)]^3}\right\}.
\]
Since $x_{n+1} = x_n+y_n$, then it follows that $|x_{n+1}|_p = |x_n+y_n|_p \leq \max\{|x_n|_p,|y_n|_p\}$.
Note, however, that we have the estimate $|y_n|_p \leq p^{-3^n}$ for any fixed prime $p$. 
So $|x_{n+1}|_p \leq \max\{1,p^{-3^n}\} =1$.
Therefore, by definition of elements in $\mathbb{Z}_p$, we have $x_{n+1}\in\mathbb{Z}_p$.
On the other hand, since $f'(x_{n+1}) = f'(x_n+y_n)$, then by Lemma \ref{lemma} it follows that
\[
f'(x_{n+1})= f'(x_n)+y_n\sum_{j=0}^{q-1} \frac{f^{(j+2)}(x_n)}{(j+1)!}y_n^j.
\]
Taking modulo $p$ on both sides of the above equation and by Lemma \ref{estimate}, we get $f'(x_{n+1})\equiv_p f'(x_n) \not\equiv_p0$. 
Moreover, since $f(x_{n+1})=f(x_n+y_n)$ and by Lemma \ref{lemma}, we have the equation
\[
f(x_{n+1}) \equiv_{p^{3^{n+1}}} f(x_n)+f'(x_n)y_n + \frac{f''(x_n)}{2}y_n^2.
\]
Expanding the right hand side of the above relation, we get
\[
\frac{[f(x_n)]^3 [f''(x_n)]^2 \left\{ 4 [f'(x_n)]^2+f(x_n)f''(x_n)\right\}}{8 [f'(x_n)]^6}=:\frac{U_n}{V_n}.
\]
Similar to what we have observed earlier, we'll obtain the congruence relation $f(x_{n+1})\equiv_{p^{3^{n+1}}}0$ for $p\neq 2$. 
On the other hand, for $p=2$ we get (as $U_n\equiv_{2^7}$ and $V\equiv_{2^3}0$) $f(x_{n+1})\equiv_{p^{4^{n+1}}}0$.
It now follows by induction that 
\[
	\forall n\in \mathbb{N}_0: \ x_n\in\mathbb{Z}_p,\quad f'(x_n)\not\equiv_p0, \quad
	f(x_n) \equiv
	\left\{
	\begin{array}{ll}
	0 \left({\rm mod} \ p^{3^n}\right)&\text{if}\ p \neq 2,\\[1em]
	0 \left({\rm mod} \ p^{4^n}\right)&\text{if}\ p=2.
	\end{array}
	\right.
\]
thereby validating our first result (i).\\

 %-------------------------- SECOND PART
\underline{Proof of ii.}
Now for the second part we need to show that, with the same assumption as in the first part, the sequence $(x_n)_{n\in\mathbb{N}_0}$ defined recursively by \eqref{recursion} converges to a unique zero $\xi\equiv_p x_0$ of $f$ in $\mathbb{Z}_p$.  
To establish this result, we first prove that $(x_n)_{n\in\mathbb{N}_0}$ is Cauchy. 
So we proceed as follows.
Note that
\begin{align*}
\left|x_{n+1}-x_n\right|_p
&=\left| x_n-\frac{f(x_n)}{f'(x_n)}-\frac{1}{2}\left\{ \frac{[f(x_n)]^2f''(x_n)}{[f'(x_n)]^3}\right\}-x_n\right|_p\\
&=\frac{\left|2f(x_n)[f'(x_n)]^2-[f(x_n)]^2f''(x_n)\right|_p}{\left|2[f'(x_n)]^3\right|_p}
=:\frac{|A|_p}{|B|_p}.
\end{align*}
From here we consider two cases: $p\neq 2$ and $p=2$.\\

\underline{Case $p\neq 2$.} If $p\neq2$, then $A\equiv_{p^{3^n}} 0$ since $f(x_n)\equiv_{p^{3^n}}0$ by (i). 
Furthermore, $B\not\equiv_p0$ since $f'(x_n)\not\equiv_p 0$ by (i) and $p\neq 2$.
Hence, from Lemma \ref{estimate}, we have the estimate $\left|x_{n+1}-x_n\right|_p \leq p^{-3^n}$.
Letting $n \rightarrow \infty$, we see that $\lim_{n\rightarrow\infty}\left|x_{n+1}-x_n\right|_p=0$, i.e., $(x_n)_{n\in\mathbb{N}_0}$ is Cauchy. 
So for the case $p\neq 2$, we have $(x_n)_{n\in\mathbb{N}_0}$ is Cauchy.\\

\underline{Case $p= 2$.} If, on the other hand, $p=2$, then $A\equiv_{p^{4^{n+1}}}0$ and $B\equiv_p 0$.
Hence, by Lemma \ref{estimate}, we have $|A|_p \leq p^{-4^n-1}$ and $|B|_p \leq p^{-1}$.
Then, $|A|_p/|B|_p \leq p^{-4^n}$ and so $|x_{n+1}-x_n|_p \rightarrow 0$ as $n\rightarrow \infty$.
Therefore, for $p=2$, $(x_n)_{n\in\mathbb{N}_0}$ is again a Cauchy sequence.

We see that, in any case, the sequence $(x_n)_{n\in\mathbb{N}_0}$ is Cauchy.
Thus, it must converges to some number in $\mathbb{Z}_p$.
Finally, the uniqueness of the zero $\xi \equiv_p x_0$ of $f$ in $\mathbb{Z}_p$ follows directly from the uniqueness of the sequence $(x_n)_{n\in\mathbb{N}_0}$. \\

 %-------------------------- THIRD PART
\underline{Proof of iii.}
As we have seen from the previous item, we have the estimate 
\[
	\forall n \in\mathbb{N}_0: 
	|x_{n+1} - x_n| \leq
	\left\{
	\begin{array}{ll}
	p^{-3^n},&\text{if}\ p \neq 2,\\[1em]
	p^{-4^n},&\text{if}\ p=2.
	\end{array}
	\right.
\]
Hence, it follows immediately that the sequence of approximants $(x_n)_{n\in\mathbb{N}_0}$ converges to a unique zero $\xi \equiv_p x_0$ of $f$ in $\mathbb{Z}_p$ cubically for $p\neq 3$ and in quartic sense for $p=2$.
This completes the proof of the theorem.
\end{proof}

\section{Conclusion}

We have considered in this work the problem of finding a $p$-adic root of a general polynomial equation $P(x)=0$ with $P(x)\in\mathbb{Z}_p[x]$.  
It was shown that the sequence $(x_n)_{n\in\mathbb{N}_0}$ defined recursively by the iterative formula of Olver's method converges to a unique zero $\xi\equiv_p x_0$ of $P$ in $\mathbb{Z}_p$.  
Consequently, the analogue of Olver's method in the $p$-adic case converges to a root in $\mathbb{Z}_p$ in cubic and quartic sense for $p\neq2$ and $p=2$, respectively. 
The result, moreover, suggests that Olver's method is more efficient in computing $2$-adic roots of polynomial equations $P(x)=0$ compare to Halley's (cf. \cite{rabago}), and Steffensen's (cf. \cite{bacani}) and Newton-Raphson's (cf \cite{rabago0}) method which has cubic and quadratic rate of convergence, respectively.

\section{Acknowledgement}

The author would like to thank Prof. Peter J. Olver of the University of Minnesota for bringing the original paper from which Olver's method first appeared to his attention.

\end{document}